\documentclass{amsart}
\usepackage{amsmath, enumerate}
\usepackage{algorithm}
\usepackage{algpseudocode}
\usepackage{amsfonts}
\usepackage{amssymb}
\usepackage{amsthm}
\usepackage{color}
\usepackage{caption}
\usepackage{graphicx}
\usepackage{url}

\newcommand{\Z}{\mathbb{Z}}
\newcommand{\gam}{\Gamma}

\newcommand{\N}{\mathbb{N}}

\newcommand{\R}{\mathbb{R}}

\newcommand{\bp}{\begin{problem}}

\newcommand{\ep}{\end{problem}}

\newcommand{\ba}{\begin{answer}}

\newcommand{\ea}{\end{answer}}

\newcommand{\ben}{\renewcommand{\theenumi}{\alph{enumi}}

\renewcommand{\labelenumi}{(\theenumi)}\begin{enumerate}}

\newcommand{\een}{\end{enumerate}}

\newtheorem{defin}{Definition}[section]
\newtheorem{thm}[defin]{Theorem}
\newtheorem{obs}[defin]{Observation}
\newtheorem{cor}[defin]{Corollary}
\newtheorem{rem}[defin]{Remark}
\newtheorem{lem}[defin]{Lemma}
\newtheorem{prop}[defin]{Proposition}

\DeclareCaptionType{result}[Table][List of Results]

\title[Expanders and RAAGs]{Expanders and right-angled Artin groups}
\begin{document}
\date{\today}
\author[R. Flores]{Ram\'{o}n Flores}
\address{Ram\'{o}n Flores, Department of Geometry and Topology, University of Seville, Spain}
\email{ramonjflores@us.es}

\author[D. Kahrobaei]{Delaram Kahrobaei}
\address{Delaram Kahrobaei, Department of Computer Science, University of York, UK, New York University, Tandon School of Engineering, PhD Program in Computer Science, CUNY Graduate Center}
\email{dk2572@nyu.edu, delaram.kahrobaei@york.ac.uk}

\author[T. Koberda]{Thomas Koberda}
\address{Thomas Koberda, Department of Mathematics, University of Virginia, Charlottesville, VA 22904}
\email{thomas.koberda@gmail.com}

\begin{abstract}
The purpose of this article is to give a characterization of families of expander graphs via right-angled Artin groups. We prove that
a sequence of simplicial graphs $\{\gam_i\}_{i\in\mathbb{N}}$ forms a family of expander graphs if and only if
a certain natural mini-max invariant arising from the cup product in the cohomology rings of the groups $\{A(\gam_i)\}_{i\in\mathbb{N}}$ agrees with the
Cheeger constant of the sequence of graphs, thus allowing us to characterize expander graphs via cohomology. This result is proved in the more general framework of \emph{vector space expanders}, a novel structure consisting of sequences of vector spaces equipped with vector-space-valued bilinear pairings which satisfy a certain mini-max condition. These objects can be considered to be analogues of expander graphs in the realm of linear algebra, with a dictionary being given by the cup product in cohomology, and in this context represent a different approach to expanders that those developed by Lubotzky-Zelmanov and Bourgain-Yehudayoff.

\end{abstract}

\maketitle

\section{Introduction}
\label{Intro}

Expander graphs, which are infinite  sequences of graphs  of bounded valence which are uniformly difficult to
disconnect, are of  fundamental importance in  discrete mathematics, graph theory, knot theory,
network theory, and statistical mechanics, and have a host of applications in computer science including to
probabilitstic computation, data organization, computational flow,
amplification of hardness, and construction of hash functions~\cite{CLG,GILVZ,HLWBAMS}. Many constructions of
graph expander families are now known, though originally explicit constructions were few
 despite the fact that their existence  is relatively easy  to prove through probabilistic methods
(see~\cite{KowalskiBook,Alon86,LPS88} for discussions of both explicit and probabilistic constructions).

In this  paper,  we provide a new perspective on graph expander families that relates them to fundamental objects in geometric group theory,
and which allows them to be probed in a novel way through linear algebraic methods.
In particular, we  characterize families of expander  graphs  through their  associated  right-angled Artin groups, and in the process define
the notion of \emph{vector space expander families}.

Recall that a \emph{simplicial graph} (sometimes known in the literature as a \emph{simple} graph) is an undirected graph with no double
edges between any pair of vertices and with no edges whose source and target coincide.
If $\gam$ is a finite simplicial graph
with vertex set $\mathrm{Vert}(\gam)$ and edge set $\mathrm{Edge}(\gam)$,
we define the \emph{right-angled Artin group} on $\gam$ by
\[A(\gam)=\langle \mathrm{Vert}(\gam)\mid [v_i,v_j]=1 \textrm{ if and only if }\{v_i,v_j\}\in \mathrm{Edge}(\gam)\rangle.\]

It is well-known that the isomorphism type of a finite  simplicial graph is uniquely  determined by the corresponding  right-angled Artin group,
and thus all the combinatorial properties one may assign to $\gam$ should be reflected in the
intrinsic algebra of $A(\gam)$~\cite{Droms87,Sabalka09,KoberdaGAFA,koberda21survey}.

 If $A(\gam)$ is given via a presentation as above (as opposed to as an abstract group), then there is a trivial way to pass between the graph
 $\gam$ and elements of the group $A(\gam)$. Indeed, the vertices of $\gam$ are then identified with the generators in the presentation,
 and the adjacency relation in $\gam$ is exactly the commutation relation among generators of $A(\gam)$. The problem with this
 perspective is that a choice of generators of $A(\gam)$ is \emph{not canonical}.
 For instance, it is possible to find a generating set of $A(\gam)$ that
 such that commutation relations between generators have nothing to do with the combinatorics of $\gam$.
 The point of this paper is to
 translate between the combinatorics of $\gam$ and the algebraic structure of $A(\gam)$
 in a way that is \emph{intrinsic} to $A(\gam)$. Specifically,
 we wish to characterize graph expander families in a canonical algebraic way, and in particular without any reference to specific generators
 of the right-angled Artin group.
Some examples of  this principle are as follows:

\begin{enumerate}
\item
$A(\gam)$ decomposes as a nontrivial direct product if  and only if $\gam$ is a nontrivial join~\cite{Servatius1989}.
\item
$A(\gam)$ decomposes as a nontrivial free product if and only if $\gam$ is disconnnected~\cite{BradyMeier01,koberda21survey}.
\item
$A(\gam)$ contains a subgroup isomorphic to a product $F_2\times F_2$ of
nonabelian free groups if and only if  $\gam$ has a full subgraph which is  isomorphic to a square~\cite{Kambites09,KK2013gt}.
\item
The poly-free length of $A(\gam)$ is two if and only if $\gam$ admits an independent set $D$ of vertices
such that every cycle in $\gam$ meets $D$ at least twice~\cite{HS2007}.
\item
$A(\gam)$ is obtained from infinite
cyclic groups through iterated  free products and  direct products if  and  only if $\gam$ contains no full subgraph
which is isomorphic to a path of length three~\cite{KK2013gt,KK2018jt}.
\item
$A(\gam)$ is a semidirect product of
two free groups of finite rank  if and only if $\gam$ is a finite tree or a finite complete bipartite graph~\cite{HS2007}.
\item
There is a finite nonabelian group acting faithfully on $A(\gam)$  by outer automorphisms if and only if  $\gam$ admits a nontrivial
automorphism~\cite{FKK2019}.
\item
A graph $\gam$ with $n$ vertices is $k$--colorable if and only if there is  a  surjective map
\[A(\gam)\to \prod_{i=1}^k F_i,\] where for $1\leq i\leq k$ the group $F_i$ is a free group of rank $m_i$, and where $\sum_{i=1}^k m_i=n$
~\cite{FKK2020a}.
\end{enumerate}

In this paper, we develop this dictionary by characterizing graph expander families through the intrinsic algebra of right-angled Artin groups.
Recall  that a family $\{\gam_i\}_{i\in\N}$ of
finite graphs is called a \emph{graph expander family}  if the
number  of  vertices in $\gam_i$ tends to  infinity as $i$ tends to  infinity, if the
 valence of each vertex of $\gam_i$ is bounded  independently of $i$,  and if a certain
 isoperimetric invariant called the~\emph{Cheeger constant}
 (or \emph{expansion constant}) of each $\gam_i$ is uniformly bounded away from zero. We  refer the reader
 to Section~\ref{sec:expander} for precise definitions. We remark that in general, graph expander families are not assumed
 to consist of simplicial graphs, though for the purposes of the algebraic dictionary we develop here,
  we will retain a blanket assumption that all graphs under consideration
 are simplicial unless explicitly noted otherwise.

The main result of the present paper is to give an intrinsic algebraic
characterization of graph expander  families via right-angled Artin groups, without
any reference to distinguished generating sets. In order to achieve this, one must define
a certain analogue $h_V$ of the Cheeger constant that can be
described from the data of the right-angled Artin group. This constant is constructed in terms of the triple  \[\{(H^1(A(\gam),L),H^2(A(\gam),L),\smile)\},\] where $H^i(A(\gam),L)$ is the $i^{th}$ cohomology group of $A(\gam)$ with coefficients in a field $L$, and $\smile$ the cup product restricted to $H^1(A(\gam),L)$ (see \ref{sec:defcheeger}). The following result, which is a central pillar of this paper, establishes the link between the two versions of the Cheeger constant:

\begin{prop}[cf.~Theorem~\ref{lem:cheeger}]\label{prop:main-technical}
Let $\gam$ be a finite simplicial graph, let $h_{\gam}$ denote the Cheeger constant of $\gam$, and let $h_V$ denote the Cheeger constant
of the triple \[\{(H^1(A(\gam),L),H^2(A(\gam),L),\smile)\}.\] Then $h_{\gam}=h_V$.
\end{prop}

Proposition~\ref{prop:main-technical} is the key in establishing a group-theoretic description of expander graphs. Our main result is therefore
as follows:

\begin{thm}\label{cor:raag}
Let $\{\gam_i\}_{i\in\N}$ be a family of finite
simplicial graphs, let $\{A(\gam_i)\}_{i\in\N}$ denote the corresponding family of right-angled Artin
groups, and let $L$ be an arbitrary field. Then $\{\gam_i\}_{i\in\N}$ is a graph expander family if and only if:
\begin{enumerate}
\item
The rank (i.e.~size of the smallest generating set) of $A(\gam_i)$ tends to infinity as $i$ tends to infinity.
\item
The rank of the centralizer of each nontrivial element of $A(\gam_i)$ is bounded independently of $i$.
\item
The Cheeger constant of the family \[\{(H^1(A(\gam_i),L),H^2(A(\gam_i),L),\smile)\}_{i\in\N}\] is bounded away from zero.
\end{enumerate}
\end{thm}

 This result is proved in the more general framework of ~\emph{vector space
 expanders} (with a precise definition in Subsection~\ref{sec:vs-exp} below).
 This is a certain sequence of  triples $\{(V_i,W_i,q_i)\}_{i\in\N}$, each of which is  defined over  a fixed field $L$, where each  $V_i$ is a
 finite dimensional vector space such that $\dim V_i\to  \infty$  as $i\to\infty$. Each $W_i$ is an $L$--vector space, and $q_i$ is a
 symmetric or  anti-symmetric $W_i$--valued bilinear pairing
 on $V_i$. The  family  $\{(V_i,W_i,q_i)\}_{i\in\N}$
 is a  vector space expander family if  the pairings $\{q_i\}_{i\in\N}$ satisfy certain linear algebraic criteria
 called \emph{bounded $q_i$--valence} and \emph{bounded  Cheeger constant} in  a uniform
 way. As mentioned already, the Cheeger constant is defined generally for the data $(V,W,q)$ (see Subsection~\ref{sec:vs-exp} for details).

 In this context, the previous theorem can be restated succinctly as follows:
 
\begin{thm}\label{thm:main}
Let $\{\gam_i\}_{i\in\N}$ be a family of
 finite simplicial graphs, and let $\{A(\gam_i)\}_{i\in\N}$ denote the corresponding family of right-angled Artin
groups. Then $\{\gam_i\}_{i\in\N}$ is a graph
 expander family if and only if \[\{(H^1(A(\gam_i),L),H^2(A(\gam_i),L),\smile)\}_{i\in\N}\] is a vector space
expander family.
\end{thm}

Observe that connectedness of the graphs in the family is not assumed as a hypothesis of the stated theorems, nor 
shall it be for us in the definition of a graph
expander family. Instead, connectedness of the graphs in both cases is a consequence of the Cheeger constant being nonzero.
We remark that whereas the cohomology  vector spaces of
a  right-angled Artin group depend  on the field over which they  are defined, the  property  of being a  graph  expander  family
or  a vector space  expander family is independent  of the choice of field.
As a further remark concerning the fields occurring in the previous results, it will become apparent to the reader that not
only can $L$ be arbitrary, but it need not be fixed as the index $i$ varies. Indeed, the numerical invariants used to define vector space
expanders are all either related to the non-degeneracy of the bilinear pairing or to dimension, both of which are blind to the intrinsic structure
of the field of definition.

The Cheeger constant of a  finite graph is an invariant that is computable  from the adjacency matrix of the graph. The Cheeger  constant
of  a vector space equipped with a  pairing  is  less obviously computable, since its  definition
quantifies over all subspaces of up to  half the  dimension of  the ambient space (see Subsection~\ref{sec:vs-exp} below).
However, the  reader  will  note  that the methods
in Section~\ref{sec:cheeger} are explicit and constructive, and they do  in fact effectively yield the Cheeger  constant  of  the relevant vector
spaces.

The  notion of a vector space expander family is more flexible than that of
a graph expander family, and  we will illustrate this
with an example of a vector space expander family which does not arise from the cohomology of the right-angled Artin groups associated
to a graph expander family. This is a reflection of the relatively lax hypotheses on the input data of a vector space expander family.
For instance, the vector space  valued bilinear pairing is more or less  arbitrary other  than being assumed to be (anti)--symmetric,
which relaxes much of the inherent  structure of the cup
product on  the cohomology of a right-angled Artin group.
The authors expect that the flexibility of vector space expanders will  contribute to their applicability.

There is another linear-algebraic version of  expanders, called \emph{dimension expanders}, which were proven to  exist by
Lubotzky--Zelmanov in the case of characteristic  zero  fields~\cite{LZ08},  and by Bourgain--Yehudayoff
in the case of  finite  fields~\cite{BY13,Bourg09}. Here, one considers a finite dimensional vector space $V$  and a collection of  $k$
linear maps  $\{T_i\colon V\to  V\}_{1\leq i\leq  k}$.
This data is  called an  \emph{$\epsilon$--dimension expander}  if for all subspaces $W\subset V$ of dimension at  most half of that  of $V$,
the dimension of \[W+\sum_{i=1}^k T_i(W)\] is at least $(1+\epsilon)\dim W$. The construction of dimension expanders
(with $\epsilon$ bounded away from zero, $k$ bounded above,  and the dimension of $V$ tending to infinity) is much harder
over finite fields than over fields of characteristic zero, whereas the constructions in the present paper are independent of the base field.
One bridge between graph expander families and dimension expanders arises from interpretation of
regular graphs of even valence as  Schreier
graphs, from which one can use finitary versions of Kazhdan's property  (T) to  construct the suitable linear  maps. The authors
do not know how to
relate dimension expanders to vector space expanders, since a general right-angled Artin group does not usually admit any
natural endomorphisms
of its first cohomology.

The paper is organized as follows. Section~\ref{sec:expander} introduces the definitions of the objects considered in this paper.
Section~\ref{sec:raag} discusses the cohomology of right-angled Artin groups, and the circle of ideas relating connectedness of graphs,
pairing--connectedness, $q$--valence, graph valence, and ranks of centralizers of elements  in a right-angled  Artin group.
 Section~\ref{sec:cheeger} establishes the main technical result of the paper, namely that  the linear-algebraic Cheeger constant associated
 to a vector space with an (anti)--symmetric bilinear pairing agrees with the Cheeger constant of a finite simplicial graph in the case that
 the vector space is the first cohomology of  the right-angled Artin group on the graph, and the bilinear pairing is  the cup product.
 Section~\ref{sec:ex} builds an example of a vector space expander family not arising from the cohomology of right-angled Artin  groups
 on a graph expander family.

\section{Graph and vector space expanders}\label{sec:expander}
In this section, we recall some relevant facts about graph expander families and define vector space expander families.

\subsection{Graph expander families}
The literature on graph expander families and  their applications is enormous.
The reader may  consult~\cite{HLWBAMS,LubBook,LPS88,KowalskiBook} and the references therein, for  example.
For the sake of brevity, we will only discuss
the combinatorial definition of an expander family.

Let $\gam$  be a finite graph, not necessarily simplicial, with vertex set $\mathrm{Vert}(\gam)$ and edge set $\mathrm{Edge}(\gam)$.
We assume that $\gam$ is undirected. If $A\subset \mathrm{Vert}(\gam)$, we write $\partial A$ for the \emph{neighbors of $A$}.
That is, $\partial A$ consists of the vertices of $\mathrm{Vert}(\gam)$ which are not contained in $A$  but which are adjacent to a vertex in $A$.

If in addition $|A|\leq  |\mathrm{Vert}(\gam)|/2$, we consider the isoperimetric invariant \[h_A=\frac{|\partial A|}{|A|}.\] The \emph{Cheeger constant}
$h_{\gam}$ is defined to be \[h_{\gam}=\min_A h_A,\] where  the minimum is taken over all subsets of $\mathrm{Vert}(\gam)$
satisfying $|A|\leq  |\mathrm{Vert}(\gam)|/2$.

Let $\{\gam_i\}_{i\in \N}$ be a sequence of connected graphs such that $|\mathrm{Vert}(\gam_i)|\to\infty$, such  that each vertex in $\gam_i$ has valence  which is
bounded independently of $i$. We say that $\{\gam_i\}_{i\in \N}$ is a \emph{graph expander family} if $\inf_i h_{\gam_i}>0$.

We note that as is well known, the bound $\inf_i h_{\gam_i}>0$ makes any connectivity assumption of the graphs $\{\gam_i\}_{i\in \N}$
redundant. Indeed, if $\gam$ is disconnected then there is a component $\Lambda$ of $\gam$ that contains at most half of the
vertices of $\gam$. Setting $A=\mathrm{Vert}(\Lambda)$, we obtain $\partial A=\varnothing$, and so $h_{\gam}=0$.

\subsection{Vector space expander families}\label{sec:vs-exp}

Throughout this section and for the rest of the paper, we fix a field  $L$  over which all  vector spaces will be defined.
All bilinear pairings are assumed to be symmetric or anti-symmetric, so that for all suitable vectors $v$  and $w$, we have $q(v,w)=\pm q(w,v)$.
Our reasons for adopting this assumption are that it mirrors an intrinsic property of the  cup product pairing, and because otherwise  the
orthogonal complement of $F$  may  be asymmetric depending on  which  side it is defined. An asymmetric orthogonal complement would
result in an unnecessary layer of subtlety and complication that  would  not enrich the theory  at hand.

\subsubsection{The Cheeger constant}\label{sec:defcheeger}
Let $\mathcal{V}$
be a collection $\{(V_i,W_i,q_i)\}_{i\in\N}$ of finite dimensional vector spaces $V_i$ equipped with vector space valued bilinear pairings
\[q_i\colon V_i\times V_i\to W_i.\]

The \emph{Cheeger constant} of $\mathcal{V}$ is defined by analogy  to  graphs. To begin, let $V$ be a fixed  finite dimensional vector
space and let \[q\colon V\times V\to W\] be a vector space valued bilinear pairing on $V$.
Let  $F\subset V$ be a vector subspace such that $0<\dim F\leq (\dim V)/2$. We
 write $C$ for the \emph{orthogonal complement} of $F$ in $V$, so that
\[C=\{v\in V\mid q(f,v)=0 \textrm{ for all } f\in F\}.\] Clearly $C$ is  a vector subspace of $V$. The
\emph{Cheeger constant} of $F$ is defined to be \[h_F=\frac{\dim V-\dim  F-\dim C+\dim (C\cap F)}{\dim F}.\] The Cheeger constant
of $V$ is defined by \[h_V=\inf_{\dim F\leq (\dim V)/2} h_F.\] We will call $h_V$ the Cheeger  constant of the triple
$(V,W,q)$. We  will suppress $W$ and $q$ from the notation for the Cheeger constant if no confusion can arise.

We note that whereas the Cheeger constant $h_V$ may appear strange at first, it is defined in such a way as to reflect the Cheeger
constant of a graph. To see this last statement illustrated more explicitly, see Lemma~\ref{lem:special}.

\subsubsection{The $q$--valence of a vector space}
Let $V$ be a  finite dimensional vector space, and let  $q$ be a vector space valued bilinear pairing on $V$.
If $\emptyset\neq S\subset V$ and $B$ is a basis for $V$,
we write \[d_B(S)=\max_{s\in S}|\{b\in B\mid q(s,b)\neq 0\}|,\quad d(S)=\min_{B\textrm{ a basis}} d_B(S),\quad
d(V)=\min_{S\textrm{ spans $V$}} d(S).\] We call $d(V)$ the \emph{$q$--valence} of $V$.

\subsubsection{Pairing--connectedness}
Let $V$ and $q$ be as before. We say that $V$ is \emph{pairing--connected} if whenever $V\cong V_0\oplus V_1$ is a nontrivial direct sum
decomposition of $V$, then there are vectors $v_0\in V_0$ and $v_1\in V_1$ such that  $q(v_0,v_1)\neq  0$.

\subsubsection{Defining vector space expanders}

We are now ready to give the definition of a vector space expander family.
\begin{defin}\label{def:vs-expand}
We say that $\mathcal{V}$ is a \emph{vector space expander family} if the following conditions are satisfied:
\begin{enumerate}
\item
We have \[\lim_{i\to\infty} \dim V_i=\infty.\]
\item
There exists an $N$ such that for all $i$, we have $d(V_i)\leq N$.
\item
We have \[h=\inf_i h_{V_i}>0.\]
\end{enumerate}
\end{defin}

The reader may  note  that the first condition is analogous  to the requirement that the number  of vertices in a family of expander graphs
tends to infinity. The  second condition is analogous to the finite valence condition in a family  of expander graphs.

As with the connectedness assumption for graph expander families, the
pairing--connectedness of a vector space $V$ is a formal consequence of $h_V>0$. Precisely, we have the following.

\begin{prop}\label{prop:cheeger-conn}
Let $(V,W,q)$ be as above, and suppose $h_V>0$. Then $V$ is pairing--connected.
\end{prop}
\begin{proof}
Suppose the contrary, so that $V=V_0\oplus V_1$ is a nontrivial splitting of $V$ witnessing the failure of pairing--connectedness. Without
loss of generality, $\dim V_0\leq\dim V/2$. Set $F=V_0$. Note then that $V_1\subset C$, the orthogonal complement of $F$.
If $C\cap F\neq 0$ then $\dim C\geq \dim V_1+\dim (C\cap F)$. It follows that \[\dim V-\dim F-\dim C+\dim(C\cap F)\leq \dim V-\dim V_0-
\dim V_1=0,\] which proves the proposition.
\end{proof}

As we will show in Section~\ref{sec:raag}, pairing--connectedness for the triple \[(H^1(A(\gam),L),H^2(A(\gam),L),\smile)\] is equivalent
to connectedness of $\gam$.

\section{Cohomology, $q$--valence and pairing--connectedness}\label{sec:raag}

In  this section we establish a  generator-free characterization of bounded valence in a graph through cohomology of
the corresponding right-angled Artin group.

\subsection{The cohomology ring of a right-angled Artin group}\label{subsec:coho}
A general reference for this section is~\cite{koberda21survey}, for instance.
Let $\gam$ be a  finite simplicial graph and $A(\gam)$ the corresponding right-angled Artin group. The group $A(\gam)$ is naturally the
fundamental group  of a locally CAT(0) cube complex, called  the \emph{Salvetti complex} $S(\gam)$ of $\gam$. The space $S(\gam)$ is a
classifying space for $A(\gam)$, so that \[H^*(S(\gam),R)\cong H^*(A(\gam),R)\] over an arbitrary ring $R$.
The complex $S(\gam)$ can be built from
the unit cube in $\R^{|\mathrm{Vert}(\gam)|}$, with the coordinate  directions being identified with the vertices of $\gam$. One includes the face
spanned by a collection of edges if the corresponding vertices span a complete subgraph of $\gam$. Finally, one takes the image inside
$\R^{|\mathrm{Vert}(\gam)|}/\Z^{|\mathrm{Vert}(\gam)|}$, so that $S(\gam)$ is a subcomplex of a torus.

With this description, it is clear that one can build $S(\gam)$ out of a collection of tori of various dimensions, one for every complete subgraph
of $\gam$, and by gluing these tori together along distinguished coordinate subtori. The  reader may compare with
the description of the Salvetti complex given in~\cite{CharneyGD}.

Let $L$ be a field, viewed as a trivial $A(\gam)$--module.
We have that \[H^*((S^1)^n,L)\cong\Lambda(L^n),\] the exterior algebra of $L^n$. Via Poincar\'e duality, coordinate subtori
of tori making up $S(\gam)$ give rise  to preferred cohomology generators
in  various degrees of the exterior algebra, and the gluing data of the subtori determines how the exterior algebras corresponding to complete
subgraphs assemble into the cohomology algebra of $S(\gam)$.

To give  slightly more detail, let $\Lambda\subset\gam$ be  a  subgraph.
For  us, a subgraph is  always \emph{full}, in the
sense that  if $\lambda_1,\lambda_2\in \mathrm{Vert}(\Lambda)$ and  $\{\lambda_1,\lambda_2\}\in \mathrm{Edge}(\gam)$ then
$\{\lambda_1,\lambda_2\}\in \mathrm{Edge}(\Lambda)$. Full subgraphs are sometimes called \emph{induced}.
It is a well-known and standard fact that $A(\Lambda)$ is  naturally a subgroup of
$A(\gam)$~\cite{CharneyGD}.
It is  not  difficult to  see  that $A(\Lambda)$ is in fact a retract of $A(\gam)$.
The homology of $A(\gam)$ is easy to compute from the Salvetti complex, and the  cohomology
with trivial coefficients in a field can be easily computed using the Universal Coefficient Theorem. Each complete subgraph
$\Lambda$ of $\gam$ gives an exterior algebra as a subring of $H^*(A(\gam),L)$ via pullback along the retraction map
$A(\gam)\to A(\Lambda)$, and a dimension count shows that this accounts for all the cohomology of $A(\gam)$.

We are mostly concerned with $H^1(A(\gam),L)$ and $H^2(A(\gam),L)$, together with  the cup product pairing  on $H^1(A(\gam),L)$.
We  remark  that the cohomology of right-angled Artin groups and related groups  with nontrivial coefficient  modules
 has been investigated  extensively
(see~\cite{Davis98,MJ05} for example), but for our  purposes  we do not  need  any machinery beyond trivial coefficients.
The next proposition follows easily from the description of the cohomology of the  Salvetti complex above, and from the structure of
exterior  algebras.

\begin{prop}\label{p:coho}
Let $\gam$ be a finite simplicial graph.
\begin{enumerate}
\item
We have isomorphisms of  vector spaces: \[H^1(A(\gam),L)\cong L^{|\mathrm{Vert}(\gam)|},\quad H^2(A(\gam),L)\cong L^{|\mathrm{Edge}(\gam)|}.\]
\item
There is a basis $\{v_1^*,\ldots,v_{|\mathrm{Vert}(\gam)|}^*\}$ for $H^1(A(\gam),L)$ which is in bijection with the set
$\{v_1,\ldots,v_{|\mathrm{Vert}(\gam)|}\}$ of vertices  of $\gam$, and
there is a basis $\{e_1^*,\ldots,e_{|\mathrm{Edge}(\gam)|}^*\}$ of  $H^2(A(\gam),L)$ which is in bijection with the set
$\{e_1,\ldots,e_{|\mathrm{Edge}(\gam)|}\}$ of edges of $\gam$.
\item
The bases in the previous item can be chosen to have the following property:
if $e=\{v_i,v_j\}\in  \mathrm{Edge}(\gam)$ then $v_i^*\smile v_j^*=\pm e^*$, and if $\{v_i,v_j\}\notin  \mathrm{Edge}(\gam)$ then $v_i^*\smile v_j^*=0$.
\end{enumerate}
\end{prop}

If $\{e_1,\ldots,e_s\}$ denotes  the set of edges of $\gam$, then Proposition~\ref{p:coho} implies that $H^2(A(\gam))$ is generated
(over any field) by the dual vectors $\{e^*_1,\ldots,e^*_s\}$, and that these vectors  are linearly independent. We fix the basis
$\{e^*_1,\ldots,e^*_s\}$ for $H^2$ once and for all, so  that if $d$ is a $2$--cohomology class then \[d=\sum_{i=1}^s\lambda_i e_i^*.\]
With respect to this fixed basis, we call the elements  $e_i^*$ for which $\lambda_i\neq 0$  the \emph{support}  of $d$, so that
$d$ is \emph{supported} on the $e_i^*$ for  which  $\lambda_i\neq 0$. We will also fix
the basis $\{v_1^*,\ldots,v_{|\mathrm{Vert}(\gam)|}^*\}$ for $H^1$ once and for all, and all computations involving cohomology classes will
implicitly be with
respect to these bases unless explicitly noted to the contrary.

\subsection{Centralizers in right-angled Artin groups}

Recall that a graph $J$ is called a \emph{join} if its complement is disconnected. Equivalently, there are two nonempty subgraphs
$J_1$ and $J_2$  of $J$ which partition the vertices of $J$, and such that every vertex in $J_1$ is adjacent to every vertex in $J_2$.
We write $J=J_1*J_2$.

Let $\gam$ be a finite  simplicial graph and let $1\neq x\in A(\gam)$ be a nontrivial element, which is expressed as a word
in the vertices $\{v_1,\ldots,v_{|\mathrm{Vert}(\gam)|}\}$ of $\gam$ and their inverses. We say that $x$  is \emph{reduced} if
it is freely reduced with respect to the operation of  commuting adjacent vertices. That is, $x$ cannot be shortened by applying moves
of the form:
\begin{itemize}
\item
Free reduction: $\cdots a\cdot v_i^{\pm1}v_i^{\mp1}\cdot b\cdots\longrightarrow \cdots a\cdot b\cdots$;
\item
Commutation of adjacent vertices: $\cdots v_i^{\pm1}v_j^{\pm 1}\cdots\longrightarrow \cdots v_j^{\pm 1}v_i^{\pm1}\cdots$,
provided $\{v_i,v_j\}$ spans an edge of $\gam$.
\end{itemize}
An element of $A(\gam)$ is nontrivial if and only if it cannot be reduced to the identity via applications of these two
moves~\cite{cartier-foata,cgw2009,hm1995}.
We say that $x$ is \emph{cyclically reduced} if all cyclic
permutations of $x$ are also  freely reduced.
The centralizer of $x$ is described by a
theorem of Servatius~\cite{Servatius1989}:

\begin{thm}\label{thm:cent}
Suppose that $x$ is nontrivial, cyclically reduced, and  has non-cyclic centralizer.
Then there is a join $J=J_1*J_2*\cdots * J_n\subset \gam$ such that $x\in A(J)<A(\gam)$, and such that $J_i$  does not decompose
as a nontrivial join for $1\leq i\leq n$. Moreover:
\begin{enumerate}
\item
The element $x$ can be  uniquely represented as a product $x_1x_2\cdots x_n$ where $x_i\in A(J_i)$.
\item
Up to re-indexing, the centralizer of $x$  is given by \[\Z^k\times A(J_{k+1})\times\cdots\times A(J_n),\] where $x_i$ is nontrivial for
$i\leq k$ and trivial for $i>k$.
\end{enumerate}
\end{thm}

Let $J=J_1*J_2*\cdots * J_n$ be a join and let $v$ be a vertex in $J_1$. Then $v$ is  adjacent to each vertex of $J_i$ for $i\geq  2$,
whence it follows that  the valence of $v$ is at least \[\sum_{i=2}^n |J_i|.\] The following consequence is now straightforward:

\begin{cor}\label{cor:valence}
Let $N$ denote the maximum valence  of a vertex in $\gam$ and let $R(x)$ denote the rank of the centralizer of a
nontrivial element of $x\in A(\gam)$. Then \[N+1=\max_{1\neq x\in A(\gam)} R(x).\]
\end{cor}

In Corollary~\ref{cor:valence}, the \emph{rank} of a group is the minimal cardinality of a set of generators.

\begin{rem}
Note that Corollary~\ref{cor:valence} gives an intrinsic bound on valence of vertices in the defining graph of a right-angled Artin group
without any reference to a set of generators.
\end{rem}

\subsection{Centralizers and $q$--valence}

Let $L$ be  a  fixed field. In this subsection we prove the following linear algebraic version of valence in a graph:

\begin{lem}\label{lem:valence}
Let $V=H^1(A(\gam),L)$, let $W=H^2(A(\gam),L)\}$, and let $q$ denote the cup product pairing
\[\smile\colon H^1(A(\gam),L)\times H^1(A(\gam),L)\to H^2(A(\gam),L)\] Then the $q$--valence
$d(V)$ coincides with the maximum valence  of a vertex in $\gam$.
\end{lem}

\begin{proof}
We write $d(\gam)$ for the maximum valence of a vertex in $\gam$. Let \[B=S=\{v_1^*,\ldots,v_{|\mathrm{Vert}(\gam)|}^*\}\] be the basis for $V$
furnished by Proposition~\ref{p:coho}. Then clearly \[d(\gam)=\max_{s\in S} |\{b\in B\mid q(s,b)\neq 0\}|,\] whence it follows that
$d(V)\leq d(\gam)$.

We now consider the reverse inequality. Note first that we need only consider sets $S$ which  are bases for $V$, since if $B$ is fixed and
if $S\subset S'$ then $d_B(S)\leq d_B(S')$.

Let $S$ be an arbitrary basis for $V$, and let $v_1$ be the vertex of $\gam$ with  highest valence. If $s\in S$ then we may write $s$ in
terms of the basis $\{v_1^*,\ldots,v_{|\mathrm{Vert}(\gam)|}^*\}$. Since $S$ forms a basis for $V$, there is some $s\in S$ such that the corresponding
coefficient for $v_1^*$ is nonzero. We  fix  such an $s$  for the remainder of the proof.

Write $\{w_1,\ldots,w_k\}$ for the vertices of $\gam$ which are adjacent to $v_1$, with corresponding duals $\{w_1^*,\ldots,w_k^*\}$, and
let $B$ be another arbitrary basis for $V$. Observe first that $q(v_1^*,w_i^*)\neq 0$ for $\{1\leq i\leq k\}$. Moreover, the set
\[\{q(v_1^*,w_i^*)\}_{1\leq i\leq k}\]  is linearly independent in  $W$. It follows that
the set \[\{q(s,w_i^*)\}_{1\leq i\leq k}\] is linearly independent in  $W$.

Thus, we may consider the linear map \[q_s\colon V\to W\] given by $q_s(v)=q(s,v)$.  Clearly this is a linear map and its image is a
vector subspace of $W$. The considerations of the previous paragraph show that the dimension of $q_s(V)$ is at least $k$, which
coincides with the valence of $v_1$ and hence
with $d(\gam)$. Suppose that there were fewer than $k$ elements $b\in B$ for which
$q(s,b)\neq 0$. Then $q_s(B)\subset W$ would span a subspace of dimension strictly less than $k$. However, $B$ is a basis, so that the
span of $q_s(B)$ coincides with $q_s(V)$, which is a contradiction. Thus, we have that  $d_B(S)\geq d(\gam)$. Since $B$ and $S$ were
arbitrary, we have $d(V)\geq d(\gam)$.
\end{proof}

\subsection{Pairing--connectedness}

In this subsection, we show that pairing--connectedness, which was already shown to be implied by positive Cheeger constant $h_V>0$
by Proposition~\ref{prop:cheeger-conn},
is equivalent to the connectedness of $\gam$ under the assumptions \[V=H^1(A(\gam),L),\quad W=H^2(A(\gam),L),\quad q=\smile.\]

\begin{lem}\label{lem:connected}
Let $\gam$ be a finite simplicial graph, let $V=H^1(A(\gam),L)$, and let $q$ be the cup product pairing on $V$. The
vector space $V$ is pairing--connected
if and only if the graph $\gam$ is connected.
\end{lem}
\begin{proof}
Let $\{v_1,\ldots,v_n\}$ be the vertices of $\gam$, so that the  dual  vectors $\{v_1^*,\ldots,v_n^*\}$ form a  basis for $V$. Suppose
that  $\gam$ is  not  connected.  Then after reindexing, we may write $B_0=\{v_1^*,\ldots,v_j^*\}$ and $B_1=\{v_{j+1}^*,\ldots,v_n^*\}$ with
$j<n$, and where there  is  no edge in $\gam$ of the form $\{v_i,v_k\}$ with $i\leq j$ and $k>j$. We  let $V_0$ be the span of $B_0$  and
$V_1$  be the span of $B_1$. Note  that  $V=V_0\oplus V_1$.
It  is clear that  if $w_0\in V_0$ and $w_1\in V_1$ then $q(w_0,w_1)=0$, so that $V$  is not pairing--connected.

Conversely, suppose  that $\gam$  is  connected, and  suppose  that $V\cong V_0\oplus V_1$ is an arbitrary  nontrivial  direct sum
decomposition. We assume for a contradiction that for all pairs $w_0\in V_0$  and $w_1\in V_1$, we have $q(w_0,w_1)=0$.
We argue by induction that either $V_0=0$ or $V_1=0$,
using a sequence $\{b_1,\ldots,b_m\}$ of vertices of $\gam$, such that each vertex  of $\gam$
appears in this sequence, and such that for all  $i<m$  we  have $\{b_i,b_{i+1}\}$  spans  an edge of $\gam$.  We  write  $b_i^*\in V$ for  the
vector dual to  the  vertex $b_i$.
Note  that it is  possible for there
to be repeats  on the list $\{b_1,\ldots,b_m\}$, since $\gam$ may not contain a Hamiltonian path.

Before starting the induction, we explain the  inductive  step.  Let  $w_0\in V_0$ and $w_1\in V_1$, and write
\[w_0=\sum_{i=1}^n\lambda_i v_i^*,\quad w_1=\sum_{i=1}^n\mu_i v_i^*.\] Suppose that $\{v_i,v_j\}$ spans an edge in $\gam$. By
expanding the cup  product $q(w_0,w_1)=0$, we see  that we must have $\lambda_i\mu_j=\lambda_j\mu_i$. If these products  are nonzero,
it follows that the pairs $(\lambda_i,\lambda_j)$ and  $(\mu_i,\mu_j)$ must  satisfy a proportionality relation
 (i.e. there is a nonzero $\alpha$
such that $(\lambda_i,\lambda_j)=(\alpha\mu_i,\alpha\mu_j)$).
The vector space
$V$ is  a free $L$--module on $\{v_1^*,\ldots,v_n^*\}$, so  that there  are vectors in $V$ whose  coefficients do not satisfy
this  proportionality  relation.
Therefore there exist vectors
\[w_0'=\sum_{i=1}^n\lambda_i'v_i^*\in V_0\quad\textrm{or}\quad w_1'=\sum_{i=1}^n\mu_i'v_i^*\in V_1\]
such that $(\lambda_i',\lambda_j')$ is not
proportional to $(\lambda_i,\lambda_j)$ or $(\mu_i',\mu_j')$ is not  proportional to
$(\mu_i,\mu_j)$. Indeed, since $V$  is spanned  by $V_0$ and $V_1$, if  there were  no such vectors in both $V_0$
and $V_1$  then every  vector in $V$  would satisfy this proportionality relation, which is not  the case.
We then see that either $q(w_0',w_1)\neq 0$ or  $q(w_0,w_1')\neq 0$, which contradicts the assumption that
$q(w_0,w_1)=0$  for all $w_0\in V_0$ and $w_1\in  V_1$. It follows that  $\lambda_i\mu_j=\lambda_j\mu_i=0$.

We can now begin  the induction. Suppose that $w_0\in V_0$ is expressed with respect  to the basis $\{v_1^*,\ldots,v_n^*\}$. After
relabeling, we may assume $v_1=b_1$ and $v_2=b_2$.
Assume  that the
coefficient $\lambda_1$ of $v_1^*=b_1^*$ is nonzero; if no such vector exists then we simply choose one in $V_1$ and proceed in the
following argument with the roles of $V_0$ and $V_1$ switched.
 Let $w_1\in V_1$  be  similarly  expressed, and suppose that the coefficient
$\mu_2$ corresponding to $b_2^*$  is nonzero. Then we must have $\lambda_1\mu_2=\lambda_2\mu_1$,  and these  products  are
both  nonzero. The argument  of the inductive step shows  that since  $V_0\oplus V_1=V$, we  cannot  have
$\lambda_1\mu_2=\lambda_2\mu_1\neq 0$. It follows  that  $\mu_2=0$. Since $w_1$ was  arbitrary, the vanishing  of this coefficient
 holds for  all vectors  in $V_1$. Again using the fact that $V_0$ and $V_1$ span $V$, there is a  vector $w_0'\in V_0$ which has
 a nonzero  coefficient $\lambda_2'$  for $b_2^*$. Arguing symmetrically  shows that the  coefficient $\mu_1$ of $b_1^*$ vanishes
 for all vectors in $V_1$.

 By induction on $m$  and using  the  fact that  each vertex of $\gam$ occurs on the list $\{b_1,\ldots,b_m\}$,
 it follows that if $w_1\in V_1$ then  all coefficients  of $w_1$ with respect to the basis $\{v_1^*,\ldots,v_n^*\}$ vanish,
 so that $V_1$ is  the zero  vector space. This contradicts  the assumption that $V\cong V_0\oplus V_1$  was  a nontrivial direct  sum
 decomposition.
\end{proof}

\section{Graph and vector space Cheeger constants}\label{sec:cheeger}
In this section we show that a vector space equipped with a vector space valued bilinear pairing can compute
the Cheeger constant of a graph, which will allow us to establish Theorem~\ref{thm:main}  and  its consequences.

\subsection{Comparing Cheeger constants}

The  main technical result of this section is the following, which provides a precise correspondence  between Cheeger constants in the
combinatorial and linear algebraic contexts:

\begin{thm}\label{lem:cheeger}
Suppose that $\gam$ is a connected simplicial graph and let $A(\gam)$
 be the corresponding right-angled Artin group. Let $h_{\gam}$ denote the
Cheeger constant of $\gam$, and let $h_V$ denote the Cheeger constant of the triple $(V,W,q)$, where $V=H^1(A(\gam),L)$, where
$W=H^2(A(\gam),L)$, and where $q$ denotes the cup product. Then $h_{\gam}=h_V$.
\end{thm}

The  proof of Theorem~\ref{lem:cheeger} is rather involved, and so will be broken up into several more manageable lemmata. We begin
by proving that the Cheeger constant associated to a subspace $F\subset V$ generated  by duals of the vertex generators is given by
the Cheeger constant associated to the corresponding subgraph. To fix notation,
let $\{v_1,\ldots,v_n\}$ denote the vertices of $\gam$, and let $\{v_1^*,\ldots,v_n^*\}$ be the corresponding dual generators of $V$.
If $B=\{v_1,\ldots,v_j\}$, we write $B^*=\{v_1^*,\ldots,v_j^*\}$ and use $\partial B$ to denote the vertices $\gam$ which do not lie in $B$
but which are adjacent to vertices in $B$.

\begin{lem}\label{lem:special}
Let $0\neq F\subset V$ be generated by $B^*=\{v_1^*,\ldots,v_j^*\}$. Then \[h_F=\frac{|\partial B|}{|B|}.\]
\end{lem}
\begin{proof}
Recall that we use the notation $C$ for the orthogonal complement of $B^*$ with respect to $q$.
The subspace  $C\subset V$ is  generated by
vertex duals $\{y_1^*,\ldots,y_m^*\}$, where for  each $i$ either $y_i\notin B\cup\partial B$ or $y_i$ is an isolated vertex  of $B$ (i.e.
$y_i$ is not adjacent to any other vertex of  $B$).

To see this, note that $\{y_1^*,\ldots,y_m^*\}\subset C$. Conversely, suppose that $x\in C$ and write
\[x=a_1v_1^*+\cdots+a_n v_n^*,\] where $a_1\neq 0$. If $v_1$ is adjacent to a vertex $w\in B$ then clearly $q(x,w^*)\neq 0$, since the
resulting cohomology class will be supported on the dual of the edge connecting  $v_1$ and $w$ (see Subsection~\ref{subsec:coho}
for a discussion of the definition of support). It follows that if
$x\in C$ then $v_1$ is either  an isolated
vertex of  $B$ or $v_1\notin B\cup\partial B$.

We now claim that \[h_F=\frac{|\partial B|}{|B|}.\]
To establish this claim, note that $C\cap F$ is generated by the duals $\{v_1^*,\ldots,v_{\ell}^*\}$ of singleton vertices
of $B$. Write $|\partial B|=k$. It follows now that \[\dim C-\dim (C\cap F)=n-|B\cup\partial B|=n-k-j.\]

We thus obtain the string of equalities
\[\frac{|\partial B|}{|B|}=\frac{k}{j}=\frac{n-j-(n-k-j)}{j}=\frac{\dim V-\dim F-\dim C+\dim(C\cap F)}{\dim F}=h_F,\] which  establishes
the lemma.
\end{proof}

The following lemma clearly implies Theorem~\ref{lem:cheeger}.
\begin{lem}\label{lem:compare}
Let $0\neq F\subset V$ be  of dimension $j$. Then there exists a subspace $F'\subset V$ of dimension $j$
with a basis contained in $\{v_1^*,\ldots,v_n^*\}$, and such that $h_{F'}\leq h_F$.
\end{lem}

Observe that in order to establish Lemma~\ref{lem:compare},
if we write $C'$  for the complement of $F'$ with respect to $q$, it suffices to show that
\[\dim C-\dim(C\cap F)\leq \dim C'-\dim(C'\cap F').\]

Proving Lemma~\ref{lem:compare} is also  rather complicated, so we will gather some preliminary results and  terminology first.
We will call a $j$--tuple $\{v^*_{i_1},\ldots,v^*_{i_j}\}$ \emph{admissible} if for each index $i_k$, there is a vector $w_{i_k}$ contained in the
linear span of \[\{v_1^*,\ldots,v_n^*\}\setminus \{v^*_{i_1},\ldots,v^*_{i_j}\}\] so that the vectors  of the form $f_{i_k}=v^*_{i_k}+w_{i_k}$ form a
basis for $F$. Such  bases for $F$ will be called \emph{admissible bases}.
Note that if $\{v^*_{i_1},\ldots,v^*_{i_j}\}$ is admissible  then the vectors $w_{i_k}$ are uniquely determined for $1\leq k\leq j$.
It is straightforward to determine  whether a tuple is admissible: indeed, express an arbitrary basis for $F$ in terms of the basis
 $\{v_1^*,\ldots,v_n^*\}$, the latter of which we view as the columns of a matrix. A tuple is admissible if and only if the corresponding $j\times j$ minor
 is invertible.

 Let \[E^*=\{v_1^*,\ldots,v_j^*\}\subset \{v_1^*,\ldots,v_n^*\}\] be admissible, and let
 $E=\{v_1,\ldots,v_j\}$ be the corresponding set of vertices. We write $\gam_E$ for the subgraph
 of $\gam$ spanned by $E$, and $E_0$ for the set of isolated vertices in $E$.
 For a given  subspace $F$, there are many possible admissible tuples $E^*$ we might consider.
 Among those, we will always focus our attention on those for which $|E_0|$ is
 minimized. Such a choice of $E^*$ may of course not be unique.

 Returning to an admissible basis for $F$, after re-indexing the  vertices of $\gam$ if necessary, we
 will fix a basis for $V$ now of the form
 \[\{f_1,\ldots,f_j,v_{j+1}^*,\ldots,v_n^*\},\] where $f_i=v_i^*+w_i$ as before. Such a basis for $V$ will be called \emph{standard relative
 to $F$}, and $E^*$ will be the  corresponding  admissible tuple.

 We  will fix the  following notation in the sequel.
 Suppose $F\subset V$ has dimension $j$. If $\{f_1,\ldots,f_j,v_{j+1}^*,\ldots,v_n^*\}$ is a standard basis of
 $V$ relative to $F$, write $F'$ for the  span of $\{v_1^*,\ldots,v_j^*\}$, write $C'$ for its
 orthogonal complement with respect to $q$, and let $Y$ denote the span of $\{v_{j+1}^*,\ldots,v_n^*\}$.

 We will in fact prove the following lemma, which implies Lemma~\ref{lem:compare}.

 \begin{lem}\label{lem:compare-ref}
If $F\subset V$ has dimension $j$ then there exists  a standard basis \[\{f_1,\ldots,f_j,v_{j+1}^*,\ldots,v_n^*\}\]
of $V$ relative to $F$ such that
if $x\in C$ and $F\cap C=0$ then $x\in C'\cap Y$, and if  $F\cap C\neq 0$ then \[x\in (C\cap  F)+(C'\cap Y).\]
 \end{lem}

 Lemma~\ref{lem:compare-ref} implies Lemma~\ref{lem:compare}, since then \[\dim C\leq \dim C'-\dim (C'\cap F')+\dim (C\cap F).\] We
 first establish it in the simpler cases where $\dim F=1$ and in  the case where there exists an admissible basis for $F$ with
 $E_0=\emptyset$.

\begin{proof}[Proof of Lemma~\ref{lem:compare-ref} when $\dim F=1$]
Clearly we may assume that $\dim V\geq 2$.
 Suppose $F$ is the span of $a\in V$. Observe that $F\subset C$.  We  write $\{f_1,v_2^*,\ldots,v_n^*\}$ for a standard basis for $V$
 relative to $F$. We have that $a$ is a nonzero multiple of $f_1$, and  $F'$ is the span of $v_1^*$. If $x\in C$ then we may write
 \[x=\lambda_1 f_1+\sum_{i=2}^n\lambda_i v_i^*.\] We write $w=x-\lambda_1 f_1$ and we assume $\lambda_m\neq 0$ for some $m\geq 2$.
 Note that  $q(f_1,x)=q(f_1,w)$. If $q(v_1^*,v_m^*)\neq 0$ then $q(f_1,x)$ has $\lambda_1\lambda_m$ as the  coefficient appearing
 before the vector dual to the edge $\{v_1,v_m\}$. So, if $x\in C$ then $q(v_1^*,v_m^*)=0$, whence it follows that $v_m^*\in C'$. Since $m$
 was chosen arbitrarily subject to the condition $\lambda_m\neq 0$, we have that $w\in C'\cap Y$, where $Y$ is the span of
 $\{v_2^*,\ldots,v_n^*\}$. This establishes the lemma in this case.
 \end{proof}

 \begin{proof}[Proof of Lemma~\ref{lem:compare-ref} when $E_0=\emptyset$]
 Let $\{f_1,\ldots,f_j,v_{j+1}^*,\ldots,v_n^*\}$ be a standard basis relative to $F$, where the admissible tuple $E^*$ satisfies
 $E_0=\emptyset$. Each component of $\gam_E$ consists of at least two vertices. We write $\{F',C', Y\}$ as before.
 Let $x\in C$ be written as \[\sum_{i=1}^j \lambda_i f_i+\sum_{i=j+1}^n \lambda_i v_i^*.\]  Suppose  first that $\lambda_m\neq 0$
 for some  $m\leq  j$.  The vertex $v_m\in E$ is adjacent to a vertex $v_k\in E$, so that $q(\lambda_m f_m,f_k)\neq 0$, whence it follows
 that $q(x,f_k)\neq 0$, contradicting the fact that $x\in C$. We conclude that $\lambda_m=0$ for $m\leq j$, so that we may write
 \[x=\sum_{i=j+1}^n \lambda_i v_i^*.\] Mimicking the proof in the case $\dim F=1$, we have that $x\in C'\cap Y$, as desired.
 \end{proof}

 Now let us  consider a standard basis \[B=\{f_1,\ldots,f_k,f_{k+1},\ldots,f_j,v_{j+1}^*,\ldots,v_n^*\}\] relative to $F$, where the vertices in
 the admissible tuple $E$
  with indices $1\leq i\leq k$ are precisely  those which are not isolated in $\gam_E$. We remind the reader that we assume
  here and henceforth that $B$ is chosen
  in such a way that $|E_0|$ is minimized.

  By the proofs of the cases of
  Lemma~\ref{lem:compare-ref} given so far, we may assume that $k<j$.
  Let  $x\in C$ as before,
  and write \[x=\sum_{i=1}^j\lambda_i f_i+\sum_{i=j+1}^k\lambda_i v_i^*.\]
  As argued in the proof in the case  $E_0=\emptyset$, we have that $\lambda_m=0$ for
  $m\leq k$.

  \begin{lem}\label{lem:iso-pair}
  Let $B$ be as above. If $k+1<m\leq j$ then $q(f_{k+1},f_m)=0$.
 \end{lem}
 \begin{proof}
 Write \[f_{k+1}=v_{k+1}^*+\sum_{s=j+1}^n\mu^{k+1}_s v_s^*,\quad f_{m}=v_{m}^*+\sum_{t=j+1}^n\mu^{m}_t v_t^*.\]
  By assumption, we have that $q(v_{k+1}^*,v_m^*)=0$, since the corresponding vertices are isolated.
  If $q(f_{k+1},f_m)\neq 0$ then  one of the three following  cases must  occur.

  \begin{enumerate}
   \item\label{n1}
  The coefficient $\mu_s^{k+1}$ is nonzero for a suitable $s>j$ with $q(v_s^*,v_m^*)\neq 0$.
  \item\label{n2}
  The  coefficient $\mu_t^m$ is  nonzero for a suitable $t>j$ with  $q(v_{k+1}^*,v_t^*)\neq 0$.
  \item\label{n3}
  We have $\mu_s^{k+1}\mu_t^m\neq \mu_t^{k+1}\mu_s^m$ for suitable indices $s,t>j$ with $s\neq  t$ and $q(v_s^*,v_t^*)\neq 0$.
  \end{enumerate}

In the first of these possibilities,
we write \[E'=(E\setminus \{v_{k+1}\})\cup \{v_s\}.\] We claim that $(E')^*$  remains admissible. This is straightforward
  to check.  Indeed, we record an $n\times n$ matrix $M$ whose columns are  labelled by $\{v_1^*,\ldots,v_n^*\}$,  whose rows are labelled
  by $\{f_1,\ldots,f_j,v_j^*,\ldots,v_n^*\}$, and whose entries are the $v_{\ell}^*$ coefficient $m_{i,\ell}$ of  the $i^{th}$ row
  basis element. We  have that the $j\times j$ block in the upper left hand corner is the identity matrix. Exchanging $v_{k+1}$ for $v_s$
  corresponds to switching the $(k+1)^{st}$  and $s^{th}$ columns of $M$. The $(k+1)^{st}$ row of the $s^{th}$ column reads
  $\mu_s^{k+1}\neq 0$. Thus after exchanging these  two columns, the upper left hand $j\times j$ block  remains invertible. Moreover,
  $q(v_s^*,v_m^*)\neq 0$, whence $v_s$ and  $v_m$ are no longer isolated vertices. It follows that $|E'_0|<|E_0|$, which contradicts
  the  minimality of $|E_0|$. Thus, the first item is ruled out. We may rule out the second of these items analogously.

  To rule out the third item, we let $E''=E\setminus \{v_{k+1},v_m\}\cup \{v_s,v_t\}$. It suffices to show that $(E'')^*$ is admissible,
  since $v_s$ and $v_t$ are adjacent in $\gam$ under the assumptions of the third item. We switch the columns with labels $k+1$ and $s$,
  and with labels $m$ and $t$. Since $\mu_s^{k+1}\mu_t^m\neq \mu_t^{k+1}\mu_s^m$, the determinant
  of the upper left hand $j\times j$ block remains nonzero. This establishes the lemma.
 \end{proof}

 In order to complete the proof  of Lemma~\ref{lem:compare-ref}, we will need to describe a process of modifying a given standard basis
 $B$ to obtain one with more advantageous features. Specifically, we will transform $B$ into a standard basis $B^{k+1}$ such that if
 $x\in C$ is expressed with respect to $B^{k+1}$, then the
  first $k+1$ coefficients  of $x$  must vanish.
  To this end, suppose  $f_r\notin C$ for $r>k$. Without loss of generality, $r=k+1$.

  By Lemma~\ref{lem:iso-pair}, we see that there is an index $m<k+1$ such that $q(f_m,f_{k+1})\neq 0$.
  Since $v_{k+1}$ is isolated, we have $q(v_{k+1}^*,v_m^*)=0$. Again we write
   \[f_{k+1}=v_{k+1}^*+\sum_{s=j+1}^n\mu^{k+1}_s v_s^*,\quad f_{m}=v_{m}^*+\sum_{t=j+1}^n\mu^{m}_t v_t^*.\]
   Observe that at least one of
   items~\ref{n1},~\ref{n2}, or~\ref{n3} in the proof of Lemma~\ref{lem:iso-pair} above must occur for this pairing to be nonzero.
   We now  proceed to modify $B$ to obtain a new standard
   basis $B^{k+1}$ as follows,  according to the reason for which $q(f_m,f_{k+1})\neq 0$. Namely:

   \begin{enumerate}
   \item
   If $\mu_t^m\neq 0$ for some index $t$ with $q(v_{k+1}^*,v_t^*)\neq 0$, then we set $B^{k+1}=B$.
   \item
   If the previous item does not hold but if there exists an index $s$ with $\mu_s^{k+1}\neq 0$ and $q(v_m^*,v_s^*)=0$ then
   we substitute
   $v_s^*$ for $v_{k+1}^*$ to obtain an admissible tuple as in Lemma~\ref{lem:iso-pair}. We then
   set $B^{k+1}$ to be the standard basis associated to the corresponding admissible tuple.
   \item
   If both of the previous items do not hold then
    at least one of the products $\mu_s^{k+1}\mu_t^m$ and $\mu_t^{k+1}\mu_s^m$
   is nonzero for suitable choices of indices $s$ and $t$ with $q(v_s^*,v_t^*)\neq 0$.
   We substitute $v_s^*$ for $v_{k+1}^*$.
   As before, the resulting tuple is admissible. We then write  $B^{k+1}$
   for the corresponding standard basis.
   \end{enumerate}

As before, these exchanges do not change the size of $|E_0|$.
   We now write  \[B^{k+1}=\{f_1^{k+1},\ldots,f_j^{k+1},e_{j+1},\ldots,e_n\},\] where indices have been renumbered after any
   substitutions. Note the following.

  \begin{obs}
  For $r\leq j$ and $r\neq k+1$, we have  that $f_r^{k+1}$ differs from $f_r$ by a (possibly zero) multiple of $f_{k+1}$,
   and $f_{k+1}^{k+1}=f_{k+1}$.
   \end{obs}

   If $x\in C$, we write it with respect to this new basis, so that
   \[x=\sum_{i=1}^j\lambda_i^{k+1} f_i^{k+1}+\sum_{i=j+1}^n\lambda_i^{k+1} v_i^*.\] The previous considerations show that
   $\lambda_i^{k+1}=0$
   for $i\leq k$.

   \begin{lem}\label{lem:k+1}
   The following hold.
   \begin{enumerate}
   \item
   If $x\in C$ is as above, then $\lambda^{k+1}_{k+1}=0$.
   \item
   For $k+1\leq r,s\leq j$, we have $q(f_r^{k+1},f_s^{k+1})=0$.
   \end{enumerate}
   \end{lem}
   \begin{proof}
   Suppose now that $\lambda_{k+1}^{k+1}\neq 0$, and consider the index $m$ as before which was chosen so that $q(f_m,f_{k+1})\neq 0$.
   Then for a suitable constant $\alpha$, we have \[q(f_m^{k+1},f_{k+1}^{k+1})=q(f_m+\alpha f_{k+1},f_{k+1})=q(f_m,f_{k+1})\neq 0.\]
   Moreover, $q(f_m^{k+1},f_{k+1}^{k+1})$ is supported  on the dual vector to the edge $\{v_m,v_{k+1}\}$ or
   $\{v_t,v_{k+1}\}$ (which was the edge $\{v_m,v_s\}$ or the edge $\{v_t,v_s\}$
   before the vertices  were re-indexed in the definition of $B^{k+1}$). No other summand making up the vector $x$
   (i.e.~ $\lambda_i f_i^{k+1}$ for $i\geq k+2$ or $\lambda_i^{k+1} v_i^*$ for $i\geq j+1$) is supported
   on $v_{k+1}^*$.
   It follows
   that  if $\lambda_{k+1}^{k+1}\neq  0$ then $q(x,f_m^{k+1})\neq 0$,
   which is a contradiction. We may therefore conclude that $\lambda^{k+1}_{k+1}=0$.

   For the second claim of the lemma, note that for \[k+1\leq r,s\leq j,\] we have $q(f_r,f_s)=0$ by Lemma~\ref{lem:iso-pair},
   which implies that  $q(f_r^{k+1},f_s^{k+1})=0$ as well since
   both of these vectors  differ from $f_r$ and  $f_s$ respectively by a  multiple  of $f_{k+1}$.
   \end{proof}

   Now suppose that $f_i^{k+1}\notin C$ for some $k+2\leq i\leq j$, and  without loss of generality we may  assume that $i=k+2$.
   Repeating the procedure for the
   construction of $B^{k+1}$, we may add multiples of $f_{k+2}^{k+1}$ to the basis vectors which  are distinct  from $f_{k+2}^{k+1}$ itself
   in order to obtain a new basis  \[B^{k+2}=\{f_1^{k+2},\ldots,f_j^{k+2}, v_{j+1}^*,\ldots,v_n^*\}.\] Since $q(f_{k+2}^{k+1},f_i^{k+2})=0$  for
   $i\geq k+1$, we must have that $q(f_r^{k+1},f_{k+2}^{k+1})\neq 0$ for some $r\leq k$. As before, if $x\in C$, we express $x$ in this
   basis with coefficients $\{\lambda_i^{k+2}\}_{1\leq i\leq n}$ and observe that the  coefficients satisfy $\lambda_i^{k+2}=0$ for $i\leq k$ and
   $\lambda^{k+2}_{k+2}=0$. It is conceivable that in the course
   of this modification we may find that $\lambda_{k+1}^{k+2}\neq 0$, a conclusion which we wish to rule out.

   \begin{lem}\label{lem:k+1-still}
   If $x\in C$ is expressed with respect to the basis  $B^{k+2}$, then we have $\lambda^{k+2}_{k+1}=0$.
   \end{lem}
   \begin{proof}
   We consider a vector $f_m^{k+1}$ which  satisfies $q(f_m^{k+1},f_{k+1}^{k+1})\neq  0$, and for suitable constants $\alpha$ and $\beta$,
   we obtain expressions \[f_m^{k+2}=f_m^{k+1}+\alpha f_{k+2}^{k+1},\qquad  f_{k+1}^{k+2}=f_{k+1}^{k+1}+\beta f_{k+2}^{k+1}.\]
   Computing, we have \[q(f_m^{k+2},f_{k+1}^{k+2})=q(f_m^{k+1},f_{k+1}^{k+1})+\beta q(f_m^{k+1},f_{k+2}^{k+1}),\] using the orthogonality
   of $f_{k+1}^{k+1}$ and $f_{k+2}^{k+1}$.

     It  follows that $q(f_m^{k+2},f_{k+1}^{k+2})$ is supported on the vector dual to the edge $\{v_{k+1},v_r\}$ for a suitable  $r$, as this was
   already true of $q(f_m^{k+1},f_{k+1}^{k+1})$.  Then, as we argued for $B^{k+1}$ in Lemma~\ref{lem:k+1},
   we have that $\lambda^{k+2}_{k+1}=0$ again.
   \end{proof}

We can now complete the argument.

 \begin{proof}[Proof of Lemma~\ref{lem:compare-ref}]
   We inductively construct a sequence of
   distinct bases  for  $V$  and  corresponding  admissible  tuples which we write as
   \[\{B^{k+2},B^{k+3},\ldots\},\qquad \{(E^{k+2})^*,(E^{k+3})^*,\ldots\},\]
   which have the property that if $x\in C$ is written with  respect to the  basis $B^{k+s}$ then
   the coefficients $\lambda^{k+s}_{\ell}$ of $f_{\ell}^{k+s}$  are trivial
   for $\ell\leq k+s$. We are  able  to construct $B^{k+s+1}$ from  $B^{k+s}$ precisely
   when there is an index $k+s\leq i\leq j$  such that $f_i^{k+s}\notin C$. Since  $F$ is  finite dimensional, the sequence  will  terminate
   after  finitely many  terms. This will happen either for $k+s=j$ or for some  $s<j-k$.

   In the first case, we see  that $C\cap F=0$. In the second case, the basis vectors $\{f_{k+s+1}^{k+s},\ldots,f_j^{k+s}\}$ are orthogonal
   to $F$. To complete the proof  of the lemma, we set $f_i=f_i^{k+s}$ for $1\leq i\leq j$, and $F'$ is the span of the associated admissible
   tuple $(E^{k+s})^*$. As in the statement of the lemma, we write $Y$ for the  span of $\{v_{j+1}^*,\ldots,v_n^*\}$.
    If $x\in C$ then \[x=\sum_{i=k+s+1}^j \lambda_i^{k+s} f_i+y\] for a suitable vector $y\in Y$. Note  that  by  assumption, we have
    $x-y\in C$, which implies that  $y\in C$. This shows  that  $y\in  C'\cap Y$, since $q(y,f_i)=0$ for all $i\leq j$ and hence
    $q(y,v_i^*)=0$ for $i\leq j$. It follows that if $C\cap F=0$ then $x=y\in C'\cap Y$, and otherwise that
    $x\in (C\cap F)+(C'\cap Y)$, which completes  the  proof.
    \end{proof}

\subsection{Proof of the main results}
Theorem~\ref{cor:raag} and Theorem~\ref{thm:main} now follow almost immediately. The size of the set of vertices of $\gam_i$ tending
to  infinity  is equivalent to the dimension of  $V_i=H^1(A(\gam))$ tending to infinity, over any field.
Bounded $q_i$--valence of $V_i$,
bounded valence of $\gam_i$, and bounded  centralizer rank in  $A(\gam)$ are all equivalent by Corollary~\ref{cor:valence} and
Lemma~\ref{lem:valence}. Finally, Theorem~\ref{lem:cheeger} implies  that the  Cheeger constant of $\gam_i$ is equal to the
Cheeger constant of the triple $(H^1(A(\gam)),H^2(A(\gam)),q)$, over any field. This establishes the main results.

\subsection{Generalizations to higher dimension}
By considering cohomology of right-angled Artin groups beyond dimension two, one can use vector space
expanders to generalize graph expanders to higher dimensions. Unfortunately,
this does not seem to give much new information, as might be expected;
indeed, the cohomology of  a  right-angled Artin group is completely determined by its behavior in dimension one and the cup product pairing
therein. This can easily be seen through a suitable generalization of Proposition~\ref{p:coho} to higher dimensional
cohomology: the cohomology  of the right-angled Artin group  $A(\gam)$
 in each dimension is determined  by the corresponding number of cells in the flag  complex of $\gam$
 (with a dimension shift), and the cup product pairing is determined by the face
 relation. The  flag complex, moreover, is completely determined by its $1$--skeleton. In particular, there does not
seem to be a  meaningful connection to more fruitful  notions  of higher dimensional expanders (cf.~\cite{LubICM}, for instance).

\section{A vector space expander family that does not arise from a graph expander family}\label{sec:ex}

In this section, we give  a method  for producing families of vector space expanders that do not arise from the cohomology rings
of right-angled Artin groups of graph expanders.

Let $\{\gam_i\}_{i\in\N}$ be a family of finite connected simplicial graphs which form a graph expander and let
$L$ be an arbitrary field. We will write
\[V_i=H^1(A(\gam_i),L),\quad W_i=H^2(A(\gam_i),L),\quad q_i=\smile,\]  where $\smile$ denotes the cup product in the cohomology ring
of the corresponding group. For each  $i$, we choose an arbitrary vertex $v^i$ of $\gam_i$.
We set  $V_i'=V_i$, and we let $W_i =W\oplus L$, where the summand $L$ is generated by a  vector
$z_i^*$. We set $q_i'=q_i\oplus q_{0,i}$, where $q_{0,i}((v^i)^*,(v^i)^*)=z_i^*$, and where $q_{0,i}$ vanishes
on inputs of all other basis vectors
arising from duals of vertices, in both arguments. That is, let $\{v_1^i,\ldots,v_n^i\}$ be the vertices of $\gam_i$, and without loss of generality
we may assume that $v^i=v_1^i$. We set $q_{0,i}((v_j^i)^*,(v_k^i)^*)=0$ unless both $v_j^i$ and $v_k^i$ are equal to $v_1^i$, and we
extend by bilinearity.

\begin{prop}\label{prop:ex}
If $\mathcal{V}'=\{(V_i',W_i',q_i')\}_{i\geq 0}$  is as above then:
\begin{enumerate}
\item
The family $\mathcal{V}'$ is a vector  space expander.
\item
The family $\mathcal{V}'$ does not arise from the cohomology of the right-angled Artin groups associated to a sequence of graphs.
\end{enumerate}
\end{prop}

The  second item of Proposition~\ref{prop:ex} means that there is no family of finite connected  simplicial graphs $\{\Lambda_i\}_{i\in\N}$
such that \[V_i'= H^1(A(\Lambda_i),L),\quad W_i'=H^2(A(\Lambda_i),L),\quad q_i'=\smile.\]

\begin{proof}[Proof of Proposition~\ref{prop:ex}]
Since $V_i'=V_i$, we have that $\dim V_i'\to\infty$. Now consider $q_i'$--valence, which we denote by $d_i$,  and we  compare with
the graph valence $d(\gam_i)$ of $\gam_i$. By setting
$B=S=(\mathrm{Vert}(\gam_i))^*$  in  the  definition of $q_i'$--valence, we see that $d_i(V)\leq d(\gam_i)+1$. Thus,  $\mathcal{V}'$ has uniformly bounded
valence. For each $i$,  the vector space $V_i'$ is  already pairing--connected with  respect to the pairing  $q_i$,
and $q_i(v,w)\neq 0$  implies  $q_i'(v,w)\neq  0$, so that
$V_i'$ is pairing--connected with respect to the pairing $q_i'$.

We  now need to estimate the Cheeger constants of $\mathcal{V}'$. We suppress the $i$ index, and write  $\{v_1^*,\ldots,v_n^*\}$
for a basis of $V'$ consisting of dual vectors of vertices of $\gam$. We assume
$v_1$ to be the distinguished vertex of $\gam$ such that $q_0(v_1^*,v_1^*)\neq 0$.
Let $0\neq F\subset V'$ be a subspace
of dimension at most $(\dim V')/2$, and let $h_0$ be the infimum of the Cheeger
constants of the family $\mathcal{V}'$ with respect to $q$, the usual cup product. We denote by
$C_q$ the orthogonal complement of $F$ with respect to $q$, by $C_0$ the orthogonal complement of $F$ with respect to $q_0$,
and by $C$ the orthogonal complement  of $F$ with respect to $q'$. Clearly, $C=C_q\cap C_0$.

Now, let $f\in F$ be written  as \[f=\sum_{i=1}^n\mu_i v_i^*,\] and let $x\in V$ be written as
\[x=\sum_{i=1}^n\lambda_i v_i^*.\] It follows  by definition
 that $q_0(v_i^*,x)=0$ for $i\neq 1$, so that  $q_0(f,x)=\lambda_1\mu_1$. Thus, the span
of $\{v_2^*,\ldots,v_n^*\}$ is always  contained in $C_0$, and
consequently $C_0$ has  dimension either $n$ or $n-1$. Thus, $\dim C$ is either equal
to $\dim C_q$ or $\dim C_q-1$. Similarly, $x\in  C\cap F$ if  and only if $x\in C_q\cap C_0\cap F$, so  that $\dim (C\cap F)$ is either equal
to $\dim (C_q\cap F)$ or $\dim (C_q\cap F)-1$.

Suppose that $\dim(C\cap F)=\dim(C_q\cap F)-1$. Then $C\neq C_q$, so  that $\dim C=\dim C_q-1$. In this case,
\[\dim C-\dim (C\cap F)=\dim C_q-1-(\dim (C_q\cap F)-1)=\dim C_q-\dim(C_q\cap F).\]

It follows that
$\dim C-\dim(C\cap F)\leq\dim C_q-\dim(C_q\cap F)$, and the difference between these is at  most $1$. Writing $N=\dim V'-\dim F$,
 the Cheeger constant of $F$ satisfies
 \[h_F=\frac{N-\dim C+\dim (C\cap F)}{\dim F}\geq\frac{N-\dim C_q+\dim (C_q\cap F)}{\dim F}.\]
This proves that the Cheeger constant of $\mathcal{V}'$ is bounded away from zero, which proves that $\mathcal{V}'$ is a vector
space expander family.

To see that $\mathcal{V}'$ does not arise from a graph expander family, we note that the cup product satisfies $v_1^*\smile v_1^*=0$, and
$q'$ is constructed so that $q'(v_1^*,v_1^*)\neq 0$. This establishes the proposition.
\end{proof}

Many variations  on the construction in this section can be carried out, which illustrates the fact that vectors space expander families are indeed
significantly more flexible than graph expander families.

\section*{Acknowledgements}
The authors thank A. Jaikin  and A. Lubotzky for helpful comments, and are grateful to the anonymous referee for helpful corrections and
suggestions.
Ram\'{o}n Flores is supported by FEDER-MEC grant MTM2016-76453-C2-1-P and FEDER grant US-1263032 from the Andalusian
Government.
Delaram Kahrobaei is supported in part by a
Canada's New Frontiers in Research Fund, under the Exploration grant entitled ``Algebraic Techniques for Quantum Security".
Thomas Koberda is partially supported  by an
Alfred P. Sloan Foundation Research Fellowship and by NSF Grants DMS-1711488 and DMS-2002596.
\bibliographystyle{amsplain}
\bibliography{ref}

\def\cprime{$'$} \def\soft#1{\leavevmode\setbox0=\hbox{h}\dimen7=\ht0\advance
  \dimen7 by-1ex\relax\if t#1\relax\rlap{\raise.6\dimen7
  \hbox{\kern.3ex\char'47}}#1\relax\else\if T#1\relax
  \rlap{\raise.5\dimen7\hbox{\kern1.3ex\char'47}}#1\relax \else\if
  d#1\relax\rlap{\raise.5\dimen7\hbox{\kern.9ex \char'47}}#1\relax\else\if
  D#1\relax\rlap{\raise.5\dimen7 \hbox{\kern1.4ex\char'47}}#1\relax\else\if
  l#1\relax \rlap{\raise.5\dimen7\hbox{\kern.4ex\char'47}}#1\relax \else\if
  L#1\relax\rlap{\raise.5\dimen7\hbox{\kern.7ex
  \char'47}}#1\relax\else\message{accent \string\soft \space #1 not
  defined!}#1\relax\fi\fi\fi\fi\fi\fi}
\providecommand{\bysame}{\leavevmode\hbox to3em{\hrulefill}\thinspace}
\providecommand{\MR}{\relax\ifhmode\unskip\space\fi MR }
\providecommand{\MRhref}[2]{%
  \href{http://www.ams.org/mathscinet-getitem?mr=#1}{#2}
}
\providecommand{\href}[2]{#2}
\begin{thebibliography}{10}

\bibitem{Alon86}
Noga Alon, \emph{Eigenvalues and expanders}, vol.~6, 1986, Theory of computing
  (Singer Island, Fla., 1984), pp.~83--96. \MR{875835}

\bibitem{Bourg09}
Jean Bourgain, \emph{Expanders and dimensional expansion}, C. R. Math. Acad.
  Sci. Paris \textbf{347} (2009), no.~7-8, 357--362. \MR{2537230}

\bibitem{BY13}
Jean Bourgain and Amir Yehudayoff, \emph{Expansion in {${\rm SL}_2(\Bbb{R})$}
  and monotone expanders}, Geom. Funct. Anal. \textbf{23} (2013), no.~1, 1--41.
  \MR{3037896}

\bibitem{BradyMeier01}
Noel Brady and John Meier, \emph{Connectivity at infinity for right angled
  {A}rtin groups}, Trans. Amer. Math. Soc. \textbf{353} (2001), no.~1,
  117--132. \MR{1675166}

\bibitem{cartier-foata}
P.~Cartier and D.~Foata, \emph{Probl\`emes combinatoires de commutation et
  r\'{e}arrangements}, Lecture Notes in Mathematics, No. 85, Springer-Verlag,
  Berlin-New York, 1969. \MR{0239978}

\bibitem{CLG}
Denis~X. Charles, Kristin~E. Lauter, and Eyal~Z. Goren, \emph{Cryptographic
  hash functions from expander graphs}, J. Cryptology \textbf{22} (2009),
  no.~1, 93--113. \MR{2496385}

\bibitem{CharneyGD}
Ruth Charney, \emph{An introduction to right-angled {A}rtin groups}, Geom.
  Dedicata \textbf{125} (2007), 141--158. \MR{2322545}

\bibitem{cgw2009}
John Crisp, Eddy Godelle, and Bert Wiest, \emph{The conjugacy problem in
  subgroups of right-angled {A}rtin groups}, J. Topol. \textbf{2} (2009),
  no.~3, 442--460. \MR{2546582}

\bibitem{Davis98}
Michael~W. Davis, \emph{The cohomology of a {C}oxeter group with group ring
  coefficients}, Duke Math. J. \textbf{91} (1998), no.~2, 297--314.
  \MR{1600586}

\bibitem{Droms87}
Carl Droms, \emph{Isomorphisms of graph groups}, Proc. Amer. Math. Soc.
  \textbf{100} (1987), no.~3, 407--408. \MR{891135}

\bibitem{FKK2019}
Ram\'{o}n Flores, Delaram Kahrobaei, and Thomas Koberda, \emph{Algorithmic
  problems in right-angled {A}rtin groups: complexity and applications}, J.
  Algebra \textbf{519} (2019), 111--129. \MR{3874519}

\bibitem{FKK2020a}
\bysame, \emph{An algebraic characterization of {$k$}--colorability}, Proc.
  Amer. Math. Soc. \textbf{149} (2021), no.~5, 2249--2255. \MR{4232214}

\bibitem{GILVZ}
Oded Goldreich, Russell Impagliazzo, Leonid Levin, Ramarathnam Venkatesan, and
  David Zuckerman, \emph{Security preserving amplification of hardness}, 31st
  {A}nnual {S}ymposium on {F}oundations of {C}omputer {S}cience, {V}ol. {I},
  {II} ({S}t. {L}ouis, {MO}, 1990), IEEE Comput. Soc. Press, Los Alamitos, CA,
  1990, pp.~318--326. \MR{1150704}

\bibitem{hm1995}
Susan Hermiller and John Meier, \emph{Algorithms and geometry for graph
  products of groups}, J. Algebra \textbf{171} (1995), no.~1, 230--257.
  \MR{1314099 (96a:20052)}

\bibitem{HS2007}
Susan Hermiller and Zoran \v{S}uni\'{c}, \emph{Poly-free constructions for
  right-angled {A}rtin groups}, J. Group Theory \textbf{10} (2007), 117--138.
  \MR{2288463}

\bibitem{HLWBAMS}
Shlomo Hoory, Nathan Linial, and Avi Wigderson, \emph{Expander graphs and their
  applications}, Bull. Amer. Math. Soc. (N.S.) \textbf{43} (2006), no.~4,
  439--561. \MR{2247919}

\bibitem{MJ05}
Craig Jensen and John Meier, \emph{The cohomology of right-angled {A}rtin
  groups with group ring coefficients}, Bull. London Math. Soc. \textbf{37}
  (2005), no.~5, 711--718. \MR{2164833}

\bibitem{Kambites09}
Mark Kambites, \emph{On commuting elements and embeddings of graph groups and
  monoids}, Proc. Edinb. Math. Soc. (2) \textbf{52} (2009), no.~1, 155--170.
  \MR{2475886}

\bibitem{KK2013gt}
Sang-Hyun Kim and Thomas Koberda, \emph{Embeddability between right-angled
  {A}rtin groups}, Geom. Topol. \textbf{17} (2013), no.~1, 493--530.
  \MR{3039768}

\bibitem{KK2018jt}
\bysame, \emph{Free products and the algebraic structure of diffeomorphism
  groups}, J. Topol. \textbf{11} (2018), no.~4, 1054--1076. \MR{3989437}

\bibitem{koberda21survey}
Thomas Koberda, \emph{Geometry and combinatorics via right-angled {A}rtin
  groups}, preprint, arXiv:2103.09342.

\bibitem{KoberdaGAFA}
\bysame, \emph{Right-angled {A}rtin groups and a generalized isomorphism
  problem for finitely generated subgroups of mapping class groups}, Geom.
  Funct. Anal. \textbf{22} (2012), no.~6, 1541--1590. \MR{3000498}

\bibitem{KowalskiBook}
Emmanuel Kowalski, \emph{An introduction to expander graphs}, Cours
  Sp\'{e}cialis\'{e}s [Specialized Courses], vol.~26, Soci\'{e}t\'{e}
  Math\'{e}matique de France, Paris, 2019. \MR{3931316}

\bibitem{LubBook}
Alexander Lubotzky, \emph{Discrete groups, expanding graphs and invariant
  measures}, Modern Birkh\"{a}user Classics, Birkh\"{a}user Verlag, Basel,
  2010, With an appendix by Jonathan D. Rogawski, Reprint of the 1994 edition.
  \MR{2569682}

\bibitem{LubICM}
\bysame, \emph{High dimensional expanders}, Proceedings of the {I}nternational
  {C}ongress of {M}athematicians---{R}io de {J}aneiro 2018. {V}ol. {I}.
  {P}lenary lectures, World Sci. Publ., Hackensack, NJ, 2018, pp.~705--730.
  \MR{3966743}

\bibitem{LPS88}
Alexander Lubotzky, Ralph Phillips, and Peter Sarnak, \emph{Ramanujan graphs},
  Combinatorica \textbf{8} (1988), no.~3, 261--277. \MR{963118}

\bibitem{LZ08}
Alexander Lubotzky and Efim Zelmanov, \emph{Dimension expanders}, J. Algebra
  \textbf{319} (2008), no.~2, 730--738. \MR{2381805}

\bibitem{Sabalka09}
Lucas Sabalka, \emph{On rigidity and the isomorphism problem for tree braid
  groups}, Groups Geom. Dyn. \textbf{3} (2009), no.~3, 469--523. \MR{2516176}

\bibitem{Servatius1989}
Herman Servatius, \emph{Automorphisms of graph groups}, J. Algebra \textbf{126}
  (1989), no.~1, 34--60. \MR{1023285 (90m:20043)}

\end{thebibliography}

\end{document}